\documentclass[a4paper,10pt]{amsart}

\usepackage[
  margin=30mm,
  marginparwidth=25mm,     % + <- Width of your marginpar
  marginparsep=2mm,       % + <- Gap between text block and marginpar
  bottom=25mm,
  ]{geometry}

\usepackage[bbgreekl]{mathbbol}
\usepackage{amsfonts}
\usepackage{latexsym,amssymb,amsthm,mathrsfs,amsmath,amscd,enumerate,enumitem, amscd,color}

\usepackage{tikz}

\DeclareSymbolFontAlphabet{\mathbb}{AMSb}
\DeclareSymbolFontAlphabet{\mathbbl}{bbold}
\newtheorem{thm}{Theorem}[section]
 
 \newtheorem{lem}[thm]{Lemma}
 \newtheorem{prop}[thm]{Proposition}
\theoremstyle{definition}

 \theoremstyle{remark}

\usepackage{hyperref}

\numberwithin{equation}{section}
\allowdisplaybreaks
\begin{document}

%\footnotetext{Last modification: \today.}

\title[]
 {Discrete Hardy spaces and heat semigroup associated with the discrete Laplacian}

\author[V. Almeida]{V. Almeida}
\address{V\'ictor Almeida, Jorge J. Betancor, Lourdes Rodr\'iguez-Mesa\newline
	Departamento de An\'alisis Matem\'atico, Universidad de La Laguna,\newline
	Campus de Anchieta, Avda. Astrof\'isico S\'anchez, s/n,\newline
	38721 La Laguna (Sta. Cruz de Tenerife), Spain}
\email{valmeida@ull.es,
jbetanco@ull.es, 
lrguez@ull.es}

\author[J. J. Betancor]{J. J. Betancor}

%\author[A.J. Castro]{A. J. Castro}
%\address{Alejandro J. Castro\newline
	%Department of Mathematics, Nazarbayev University, \newline 010000 Astana, Kazakhstan}
%\email{alejandro.castilla@nu.edu.kz}

%\author[J.C. Fari\~na]{J. C. Fari\~na}

\author[L. Rodr\'{\i}guez-Mesa] {L. Rodr\'{\i}guez-Mesa}

\thanks{The authors are partially supported by MTM2016-79436-P}



\date{}

%%% ----------------------------------------------------------------------

\begin{abstract}
In this paper we study the behavior of some harmonic analysis operators associated with the discrete Laplacian $\Delta_d$ in discrete Hardy spaces $\mathcal H^p(\mathbb Z)$. We prove that the maximal operator and the Littlewood-Paley $g$ function defined by the semigroup generated by $\Delta_d$ are bounded from $\mathcal H^p(\mathbb Z)$ into $\ell^p(\mathbb Z)$, $0<p\leq 1$. Also, we establish that every $\Delta_d$-spectral multiplier of Laplace transform type is a bounded operator from $\mathcal H^p(\mathbb Z)$ into itself, for every $0<p\leq 1$.
\end{abstract}

%%% ----------------------------------------------------------------------
\maketitle
%%% ----------------------------------------------------------------------

\section{Introduction}

Coifman and Weiss (\cite{CW}) defined Hardy spaces $H^p(X)$ when the underlying space $X$ is a space of homogeneous type. They extended the atomic decomposition theory for the classical  Hardy spaces to this more general setting. The set $\mathbb Z$ of integer numbers endowed with the usual distance defined by the absolute value and the counting measure $\mu$ is a space of homogeneous type. Discrete Hardy spaces were mentioned in  \cite[p. 622]{CW} as an example of the general theory in  \cite{CW}.

Boza and Carro (\cite{BC1}, \cite{BC3} and \cite{BC2}) characterized Hardy spaces $H^p(\mathbb Z)$, $0<p\leq 1$, by using some maximal operators associated with a discretization of the classical Poisson integrals and approximations of the identity. Also, they described discrete Hardy spaces by using atoms and the discrete Hilbert transform. Boza and Carro established and took advantage of connections between discrete and continuous (classical) settings. Kanjin and Satake (\cite{KS}) developed molecular characterizations of discrete Hardy spaces. Komori (\cite {Ko}) proved that every molecule can be decomposed in atoms without using the properties of the classical Hardy spaces $H^p(\mathbb R)$. A discrete version for $H^p(\mathbb Z)$ of the weak factorization results for Hardy spaces $H^p(\mathbb R)$ due to Coifman, Rochberg and Weiss (\cite{CRW}) and Miyachi (\cite{Mi}), was proved by Boza (\cite{Bo}). Eoff (\cite[Theorem 1]{Eo}) established that, for every $0<p\leq 1$, the discrete Hardy space $H^p(\mathbb Z)$ is isomorphic to the Paley-Wiener space $E^p$ that consists of all those entire functions of exponential type $\pi$ such that $\int_{\mathbb R}|f(x)|^pdx <\infty$. Chen and Fang (\cite{CF}) extended Eoff's result to higher dimensions when $p=1$. 

The discrete Laplacian $\Delta_d$ on $\mathbb Z$ is defined by.
$$(\Delta_d f)(n)=-f(n+1)+2f(n)-f(n-1),\;\;\;\;n\in\mathbb Z,$$
where $f$ is a complex function defined on $\mathbb Z$. For every $0<p\leq\infty$, we denote by $\ell^p(\mathbb Z)$ the usual Lebesgue space on $\mathbb Z$ with respect to the counting measure $\mu$. The operator $\Delta_d$ is bounded from $\ell^p(\mathbb Z)$ into itself, for every $0<p\leq\infty$, and it is a nonnegative operator in $\ell^2(\mathbb Z)$.

We define the function
$$G(n,t)=e^{-2t}I_n(2t),\quad n\in\mathbb Z\;\mbox{and}\; t>0.$$
Here, for every $n\in\mathbb Z$, $I_n$ represents the modified Bessel function of the first kind and order $n$. The main properties of $I_n$ can be encountered in \cite[Chapter 5]{Le}.

For every $t>0$ we consider the convolution operator $W_t$ defined by
$$W_t(f)(n)=\sum_{m\in\mathbb Z}G(n-m,t)f(m),\quad n\in\mathbb Z,$$
for every $f\in\ell^p(\mathbb Z)$, $1\leq p\leq\infty$. In \cite[Proposition 1]{CGRTV} it was proved that the uniparametric family $\{W_t\}_{t>0}$ is a positive Markovian diffusion semigroup in the Stein's sense (\cite{StLP}) in $\ell^p(\mathbb Z)$, $1\leq p\leq\infty$. Moreover, $W_t=e^{-t\Delta_d}$, $t>0$, that is, $-\Delta_d$ is the infinitesimal generator of $\{W_t\}_{t>0}$.

We now define the operators that we will study on Hardy spaces $H^p(\mathbb Z)$, $0<p\leq 1$.
The maximal operator $W_*$ associated with $\{W_t\}_{t>0}$ is defined by
$$W_*(f)=\sup_{t>0}|W_t(f)|.$$
The (vertical) Littlewood-Paley $g$-function for $\{W_t\}_{t>0}$ is given by
\begin{equation}\label{funciong}
g(f)(n)=\left(\int_0^\infty |t\partial_tW_t(f)(n)|^2\frac{dt}{t}\right)^{1/2},\;\;\;\;n\in\mathbb Z,
\end{equation}
for every $f\in\ell^p(\mathbb Z)$, $1\leq p<\infty$. If $f\in\ell^1(\mathbb Z)$ the Fourier transform $\mathcal F_\mathbb{Z}(f)$ of $f$ is defined by
$$\mathcal F_\mathbb{Z}(f)(\theta)=\sum_{n\in\mathbb Z}f(n)e^{in\theta},\quad \theta\in(-\pi,\pi).$$
The Fourier transform $\mathcal F_\mathbb{Z}(f)$ can be extended from $\ell^1(\mathbb Z)$ to $\ell^2(\mathbb Z)$ as an isometry from $\ell^2(\mathbb Z)$ into $L^2(-\pi,\pi)$. Moreover, the inverse operator $\mathcal F_\mathbb{Z}^{-1}$  of $\mathcal F_\mathbb{Z}$ is defined by
$$\mathcal F_\mathbb{Z}^{-1}(\varphi)(n)=\frac{1}{2\pi}\int_{-\pi}^{\pi}\varphi(\theta)e^{-in\theta}d\theta,\;\;\;\;\varphi\in L^2(-\pi,\pi).$$
It is no hard to see that, for every $\theta\in(-\pi,\pi)$, $\Delta_d(e_\theta)=2(1-\cos \theta)e_\theta$, where $e_\theta(n)=e^{in\theta}$, $n\in\mathbb Z$. Also, for every $f\in\ell^2(\mathbb Z)$, 
$$\mathcal F_\mathbb{Z}(\Delta_d(f))(\theta)=2(1-\cos \theta)\mathcal F_\mathbb{Z}(f)(\theta),\quad \theta\in(-\pi,\pi).$$
If $\mathfrak m\in L^\infty(0,\infty)$ the spectral multiplier $T_{\mathfrak m}$ associated with $\Delta_d$ defined by $\mathfrak m$ takes the following form
$$T_{\mathfrak m} f=\mathcal F_\mathbb{Z}^{-1}[\mathfrak m(2(1-\cos \theta))\mathcal F_\mathbb{Z}(f)],\quad f\in\ell^2(\mathbb Z).$$
Thus, $T_{\mathfrak m}$ is bounded from $\ell^2(\mathbb Z)$ into itself.

We now assume that $\mathfrak m\in L^\infty(0,\infty)$ and that there exists $\Psi\in L^\infty(0,\infty)$ such that 
$$
\mathfrak m(\lambda)=\lambda\int_0^\infty e^{-\lambda t}\Psi(t)dt,\quad \lambda\in(0,\infty).
$$
The multiplier $T_{\mathfrak m}$ is now called Laplace transform type multiplier (see \cite{StLP}).

According to \cite{StLP} the operators $W_*$, $g$ and $T_{\mathfrak m}$ are bounded from $\ell^p(\mathbb Z)$ into itself for every $1<p<\infty$. In \cite{CGRTV} it was proved by using Calder\'on-Zygmund theory for Banach valued singular integrals that $W_*$ and $g$ are bounded operators from  $\ell^p(\mathbb Z,\omega)$ into itself, for every $1<p<\infty$ and $\omega\in A_p(\mathbb Z)$, and from $\ell^1(\mathbb Z,\omega)$ into $\ell^{1,\infty}(\mathbb Z,\omega)$, for every  $\omega\in A_1(\mathbb Z)$. Here if $\omega$ is a nonnegative sequence and $1\leq p<\infty$, we say that a complex sequence $f\in\ell^p(\mathbb Z,\omega)$ when
$$\|f\|_{p;\omega}:=\left(\sum_{n\in\mathbb Z}|f(n)|^p\omega(n)\right)^{1/p}<\infty.$$
On $\ell^p(\mathbb Z,\omega)$ we consider the norm $\|\cdot\|_{p;\omega}$. Also, if $\omega$ is a nonnegative sequence, we say that a complex sequence $f\in\ell^{1,\infty}(\mathbb Z,\omega)$ when
$$\|f\|_{1,\infty;\omega}:=\sup_{\lambda >0}\lambda\omega\left(\{n\in\mathbb Z\;:\;|f(n)|>\lambda\}\right)<\infty,$$
where $\omega(E)=\sum_{n\in\mathbb Z}W(n)$, for every $E\subset\mathbb Z$.
On $\ell^{1,\infty}(\mathbb Z,\omega)$ we consider the quasinorm $\|\cdot\|_{1,\infty;\omega}$.

For every $1\leq p<\infty$, by $A_p(\mathbb Z)$ we denote the Muckenhoupt class of weights. A nonnegative sequence $\omega\in A_p(\mathbb Z)$ (\cite[\S 8]{HMW}) when there exists $C>0$ such that, for every $M,N\in\mathbb Z$, $M\leq N$,
$$\left(\sum_{k=M}^N\omega(k)\right)\left(\sum_{k=M}^N\omega(k)^{-1/(p-1)}\right)^{p-1}\leq C(N-M+1)^p,\;\;\;\;1<p<\infty,$$
and
$$\left(\sum_{k=M}^N\omega(k)\right)\sup_{M\leq k\leq N}\omega(k)^{-1}\leq C(N-M+1),\;\;\;\;p=1.$$

The discrete Hilbert transform $H_d$ is defined by
$$H_d(f)(n)=\sum_{m\in\mathbb Z}\frac{f(m)}{n-m+1/2},$$
when $f\in C_0(\mathbb Z)$, where $ C_0(\mathbb Z)$ consists of all those complex sequences $f$ such that $f(n)=0$ when $|n|\geq m$ for certain $m\in\mathbb N$. It is well-known that if $1<p<\infty$, $H_d$ can be extended to $\ell^p(\mathbb Z,\omega)$ as a bounded operator from $\ell^p(\mathbb Z,\omega)$ into $\ell^p(\mathbb Z,\omega)$ if and only if $\omega\in A_p(\mathbb Z)$, and $H_d$ can be extended to $\ell^1(\mathbb Z,\omega)$ as a bounded operator from $\ell^1(\mathbb Z,\omega)$ into $\ell^{1,\infty}(\mathbb Z,\omega)$ if and only if $\omega\in A_1(\mathbb Z)$ (\cite{HMW}). In \cite[\S 6]{CGRTV} it is proved that $H_d$ can be seen as a Riesz transform in the sense os Stein (\cite{StLP}) associated to the discrete Laplacian.

In this paper we prove that the maximal operator $W_*$ and the Littlewood-Paley function $g$ are bounded from the Hardy space $H^p(\mathbb Z)$ into $\ell^p(\mathbb Z)$, and the Laplace transform type multiplier $T_{\mathfrak m}$ is bounded from $H^p(\mathbb Z)$ into itself, for every $0<p\leq 1$.

As it was mentioned Hardy spaces $H^p(\mathbb Z)$ were considered in \cite[p. 622]{CW} as an special case of the general theory developed in \cite{CW}. Boza and Carro (\cite{BC1} and \cite{BC2}) characterized Hardy spaces $H^p(\mathbb Z)$ in different ways.

We recall some definitions and properties of the discrete Hardy spaces. Let $\alpha >0$. The sequence space $\mathcal L_\alpha$ consists of all those complex sequences $\mathfrak a$ such that 
$$\|\mathfrak a\|_{\mathcal L_\alpha}:= \sup_{n,m\in\mathbb Z,\;n\neq m}\frac{|\mathfrak a(n)-\mathfrak a(m)|}{|n-m|^\alpha}<\infty.$$
$\mathcal L_\alpha$ is endowed with the quasinorm $\|\cdot\|_{\mathcal L_\alpha}$.

Let $p,q\in(0,\infty)$ such that $0<p<q$ and $p\leq 1\leq q\leq\infty$. A complex sequence $\mathfrak b$ is said to be a $(p,q)$-atom when there exist $n_0\in\mathbb Z$ and $r_0\geq 1$ satisfying  that
\begin{enumerate}
\item[(i)] The support of $\mathfrak b$ is contained in the ball $B_{\mathbb Z}(n_0,r_0)$;
\item[(ii)] $\|\mathfrak b\|_q\leq \mu\left(B_{\mathbb Z}(n_0,r_0)\right)^{1/q-1/p}$, where $1/q$ is understood to be $0$ when $q=\infty$;
\item[(iii)] $\sum_{n\in\mathbb Z}\mathfrak b(n)=0$.
\end{enumerate}
Suppose that $0<p<1$. If $\mathfrak b$ is a $(p,q)$-atom, then $\mathfrak b$ defines an element $L_{\mathfrak b}$ of the dual space $(\mathcal L_{1/p-1})'$ of $\mathcal L_{1/p-1}$ as follows
$$\langle L_{\mathfrak b},\mathfrak a\rangle=\sum_{n\in\mathbb Z}\mathfrak a(b)\mathfrak b(n),\;\;\;\;\mathfrak a\in\mathcal L_{1/p-1}.$$

Let $q\geq 1$. The space $H^{p,q}(\mathbb Z)$ is the subspace of the dual space $(\mathcal L_{1/p-1})'$  of $\mathcal L_{1/p-1}$ consisting of those linear functionals $h$ defined in $\mathcal L_{1/p-1}$ such that there exist, for every $j\in\mathbb N$, a $(p,q)$-atom $\mathfrak b_j$ and $\lambda_j>0$ satisfying that $\sum_{j\in \mathbb{N}} \lambda_j^p<\infty$ and $h=\sum_{j\in \mathbb{N}}\lambda_j\mathfrak b_j$, where the series converges in $(\mathcal L_{1/p-1})'.$ For every $h\in H^{p,q}(\mathbb Z)$ we define $\|h\|_{H^{p,q}(\mathbb Z)}$ as follows
$$\|h\|_{H^{p,q}(\mathbb Z)}=\inf\Big(\sum_{j\in \mathbb{N}}\lambda_j^p\Big)^{1/p},$$
where the infimum is taken over all those sequences $(\lambda_j)_{j\in \mathbb{N}}$ of nonnegative real numbers such that $\sum_{j\in \mathbb{N}} \lambda_j^p<\infty$ and $h=\sum_{j\in \mathbb{N}}\lambda_j\mathfrak b_j$ in $(\mathcal L_{1/p-1})'$, where $\mathfrak b_j$ is a $(p,q)$-atom, for every $j\in\mathbb N$.

Suppose that $h=\sum_{j\in \mathbb{N}} \lambda_j\mathfrak b_j$ in the sense of convergence in $(\mathcal L_{1/p-1})'$, where, for every $j\in\mathbb N$, $\lambda_j>0$ and $\mathfrak b_j$ is a $(p,q)$-atom, and $\sum_{j\in \mathbb{N}}\lambda_j^p<\infty$. By using H\"older inequality we can see that $\|\mathfrak b_j\|_p\leq 1$, $j\in\mathbb N$. Then, for every $n,m\in\mathbb N$, $n<m$,
$$\Big\|\sum_{j=n}^m \lambda_j\mathfrak b_j\Big\|_p^p\leq \sum_{j=n}^m \lambda_j^p\|\mathfrak b_j\|_p^p\leq \sum_{j=n}^m \lambda_j^p.$$
Hence, the series $\sum_{j\in \mathbb{N}} \lambda_j\mathfrak b_j$ converges in $\ell^p(\mathbb Z)$. We write $H=\sum_{j\in \mathbb{N}}\lambda_j\mathfrak b_j$ in $\ell^p(\mathbb Z)$. We have that
$$H(n)=\sum_{j\in \mathbb{N}}\lambda_j\mathfrak b_j(n),\quad n\in\mathbb Z.$$
We consider, for every $n\in\mathbb Z$, $\mathfrak b^n=(\mathfrak b^n(m))_{m\in\mathbb Z}$, where $b^n(n)=1$ and $b^n(m)=0$,
 $m\in\mathbb Z$, $m\neq n$. It is clear that $ \mathfrak b^n\in\mathcal L_{1/p-1}$, $n\in\mathbb Z$. Then,
 $$h(n):=\lim_{m\rightarrow\infty}\sum_{j=0}^m\lambda_j\sum_{n\in\mathbb Z}\mathfrak b_j(m)\mathfrak b^n(m)=\sum_{j\in \mathbb{N}}\lambda_j\mathfrak b_j(n),\quad n\in\mathbb Z.$$
It follows that $H=h$. Hence, $H^{p,q}(\mathbb Z)$ is contained in $\ell^p(\mathbb Z)$. Moreover, this inclusion is continuous.

Let $q>1$. The space $H^{1,q}(\mathbb Z)$ consists of all $h\in\ell^1(\mathbb Z)$ such that $h=\sum_{j\in \mathbb{N}} \lambda_j\mathfrak b_j$ in $\ell^1(\mathbb Z)$, where, for every $j\in\mathbb N$, $\lambda_j>0$ and $\mathfrak b_j$ is a $(1,q)$-atom, and $\sum_{j\in \mathbb{N}}\lambda_j<\infty$. For every $h\in H^{1,q}(\mathbb Z)$,
$$\|h\|_{H^{1,q}(\mathbb Z)}=\inf\sum_{j\in \mathbb{N}} \lambda_j,$$
where the infimum is taken over all those sequences $(\lambda_j)_{j\in \mathbb{N}}$ of nonnegative real numbers such that $\sum_{j\in \mathbb{N}}\lambda_j<\infty$ and $h=\sum_{j\in \mathbb{N}}\lambda_j\mathfrak b_j$ in $\ell^1(\mathbb Z)$, being $\mathfrak b_j$ a $(1,q)$-atom, for every $j\in\mathbb N$.

According to \cite[Theorem A]{CW} $H^{p,q}(\mathbb Z)=H^{p,\infty}(\mathbb Z)$ algebraically and topologically, provided that $0<p<q$ and $p\leq 1\leq q\leq\infty$.

By proceeding as in the continuous case, that is, in the $H^{p}(\mathbb R)$ case (see \cite[Theorem 7.7 and the following ones]{GCRF} and \cite[p. 598]{CW}), from \cite[Propositions 3 and 4]{CGRTV} we can prove the following result.

\begin{thm}\label{Th1.1}
Let $1/2<p\leq 1$. Then, the operators $W_*$ and $g$ are bounded from $H^{p,\infty}(\mathbb Z)$ into $\ell^p(\mathbb Z)$.
\end{thm}

In order to extend this property to values of $p\in(0,1/2]$ we consider the discrete Hardy spaces studied in \cite{BC1} and \cite{BC2}.

Let $0<p<q$ and $p\leq 1\leq q\leq\infty$. We say that a complex sequence $\mathfrak b$ is a $(H,p,q)$-atom when there exist $n_0\in\mathbb Z$ and $r_0\geq 1$ such that,
\begin{enumerate}
\item[(i)] The support of $\mathfrak b$ is contained in the ball $B_{\mathbb Z}(n_0,r_0)$;
\item[(ii)] $\|\mathfrak b\|_q\leq \mu\left(B_{\mathbb Z}(n_0,r_0)\right)^{1/q-1/p}$, where $1/q$ is understood to be $0$ when $q=\infty$;
\item[(iii)] $\sum_{n\in\mathbb Z}n^\alpha\mathfrak b(n)=0$, for every $\alpha\in\mathbb N$ such that $\alpha\leq 1/p-1$.
\end{enumerate}

We define the Hardy space $\mathcal H^{p,q}(\mathbb Z)$ in the same way that $H^{p,q}(\mathbb Z)$ where $(H,p,q)$-atoms replace $(p,q)$-atoms. We have that $\mathcal H^{p,q}(\mathbb Z)=\mathcal H^{p,\infty}(\mathbb Z)$ algebraically and topologically.

Hardy spaces  $\mathcal H^{p,q}(\mathbb Z)$ can be characterized by using discrete Hilbert transform. In \cite[Theorems 3.10 and 3.14]{BC1} it was established that a complex sequence $h$ is in $\mathcal H^{p,\infty}(\mathbb Z)$ if and only if $h\in\ell^p(\mathbb Z)$ and $H_d(h)\in\ell^p(\mathbb Z)$. Moreover, for every $h\in\mathcal H^{p,\infty}(\mathbb Z)$, the quantities $\|h\|_{\mathcal H^{p,\infty}(\mathbb Z)}$ and $\|h\|_p+\|H_d(h)\|_p$ are equivalent. In \cite{BC1} and \cite{BC2} the space $\mathcal H^{p,\infty}(\mathbb Z)$ is also characterized by using maximal operators and square functions.

For every $0<p\leq 1$ we write $\mathcal H^p(\mathbb Z)$ for naming $\mathcal H^{p,\infty}(\mathbb Z)$.

The main results of this paper are the following ones.

\begin{thm}\label{Th1.2}
Let $0<p\leq 1$. The operators $W_*$ and $g$ are bounded from $\mathcal H^p(\mathbb Z)$ into $\ell^p(\mathbb Z)$.
\end{thm}

\begin{thm}\label{Th1.3}
Assume that $\mathfrak m\in L^\infty(0,\infty)$ and that $\mathfrak m(\lambda)=\lambda\int_0^\infty e^{-\lambda t}\Psi(t)dt$, $\lambda\in(0,\infty)$, where $\Psi\in L^\infty(0,\infty)$. Then, the  operator $T_{\mathfrak m}$ is bounded from $\ell^p(\mathbb Z,\omega)$ into itself, for every $1<p<\infty$ and $\omega\in A_p(\mathbb Z)$, from $\ell^1(\mathbb Z,\omega)$ into $\ell^{1,\infty}(\mathbb Z,\omega)$, when $\omega\in A_1(\mathbb Z)$, and from  $\mathcal H^p(\mathbb Z)$ into itself for every $0<p\leq 1$.
\end{thm}

The multipliers of Laplace transform type $T_{\mathfrak m}$ can be seen as special cases of the Fourier multipliers of Marcinkiewicz type considered in \cite{KS}. Then, we can deduce that $T_{\mathfrak m}$ defines a bounded operator from  $\mathcal H^p(\mathbb Z)$ into itself from \cite[Theorem 3]{KS}. In the proof of \cite[Theorem 3]{KS} Plancherel theorem for Fourier transform plays a key role. Our proof (see section 5) of $\mathcal H^p(\mathbb Z)$-boundedness of $T_{\mathfrak m}$ is different from the one presented in \cite[Theorem 3]{KS}. We apply Proposition 2.8 and we do not use Plancherel theorem. Our procedure is more flexible because for instance it could be used in a weighted setting (see, for instance, \cite[Theorem 4]{LL}). The study of discrete Hardy spaces with weights will be addressed in a future work.

The proofs of \cite[Propositions 3 and 4]{CGRTV} rely on the following integral representation of the modified Bessel function of first kind that is known as Schl\"afli's integral representation of Poisson type for $I_\nu$ (\cite[(5.10.22)]{Le}),
\begin{equation}\label{F1}
I_\nu(z)=\frac{z^\nu}{\sqrt{\pi}2^\nu\Gamma(\nu+1/2)}\int_{-1}^1e^{-zs}(1-s^2)^{\nu-1/2}ds,\quad |\mbox{arg}\;z|<\pi,\;\mbox{and}\;\nu>-1/2.
\end{equation}
In order to prove Theorems \ref{Th1.2} and \ref{Th1.3} we wil use, instead of (\ref{F1}), the following integral representation (\cite[p. 456]{PBM}),
\begin{equation}\label{F2}
I_n(t)=\frac{1}{\pi}\int_0^\pi e^{t\cos\theta}\cos(n\theta)d\theta,\;\;\;t>0,\;\mbox{and}\;n\in\mathbb Z.
\end{equation}
After writing the operators $W_*$, $g$ and $T_{\mathfrak m}$ as Banach valued singular integrals of convolution type in the homogeneous group $(\mathbb Z,|\cdot|,\mu)$, we have to estimate derivatives of the respective kernels of those singular integrals. The integral representation (\ref{F2}) plays a key role to obtain the mentioned estimates. 

This paper is organized as follows. After this introduction in Section 2 we establish some auxiliary results that will be useful in the sequel. Theorem \ref{Th1.2} is proved in Section 3 for $W_*$ and in Section 4 for $g$. A proof for Theorem \ref{Th1.3} is written in Section 5.

Throughout this paper by $C$ and $c$ we always represent positive constants that can change from one line to the other one, and by $E[\alpha]$ we mean the highest integer number less or equal to $\alpha\in\mathbb R$.

\section{Auxiliary results}

In this section we establish some properties that will be useful in the sequel.
                      
The semigroup of operators $\{W_t\}_{t>0}$ generated by $-\Delta_d$ in $\ell^p(\mathbb Z)$ is defined by
$$                           
W_t(f)(n)=\sum_{m\in \mathbb{Z}}G(n-m,t)f(m),\quad t>0,
$$
for every $f\in \ell^p(\mathbb Z)$, $1\leq p\leq \infty$, where $G(m,t)=e^{-2t}I_m(2t)$, $t>0$ and $m\in \mathbb{Z}$.

According to (\ref{F2}) we have that
$$
G(m,t)=\frac{1}{\pi}\int_0^\pi \phi_t(\theta )\cos (m\theta)d\theta, \quad t>0\mbox{ and }m\in \mathbb{Z},
$$
where $\phi _t(\theta)=e^{-2t(1-\cos \theta)}$, $t>0$ and $\theta \in (0,\pi)$.

From now on, for every $k\in \mathbb{N}$, we consider
$$
h_k(\theta )=\left\{\begin{array}{ll}
					\sin \theta,&\mbox{ if }k\mbox{ is odd},\\
                    \cos \theta,&\mbox{ if }k\mbox{ is even},
              \end{array}
\right. \quad \theta \in (0,\pi ).
$$
Firstly we obtain a representation for the derivative of $\phi _t$, $t>0$.

\begin{lem}\label{Lem2.1}
Let $k\in\mathbb N$, $h\geq 1$. We have that, for every $t>0$ and $\theta\in (0,\pi)$,
\begin{align*}
\partial_\theta ^k\phi _t(\theta)&=\phi_t(\theta)\sum_{\substack{(m_1,...,m_k)\in\mathbb{N}^n\\m_1+2m_2+...+km_k=k}}a_{m_1,...,m_k}t^{m_1+...m_k}(\cos \theta)^{\alpha_{m_1,...,m_k}}(\sin \theta)^{\beta_{m_1,...,m_k}},
\end{align*}
where, for each $(m_1,...,m_k)\in\mathbb{N}^n$ such that $m_1+2m_2+...+km_k=k$,
$$
\alpha _{m_1,...,m_k}=\sum_{j=1}^{E[k/2]}m_{2j},
$$
$$
\beta _{m_1,...,m_k}=\sum_{j=1}^{E[(k+1)/2]}m_{2j-1},
$$
and
$$
a_{m_1,...,m_k}=(-1)^{\sum_{j=1}^{E[k/2]}jm_{2j}+\sum_{j=1}^{E[(k+1)/2]}jm_{2j-1}}\frac{ 2^{m_1+...+m_k}k!}{m_1!1!^{m_1}m_2!2!^{m_2}\cdots m_k!k!^{m_k}}.
$$
\end{lem}
\begin{proof}
Let $t\in (0,\infty )$. We define $\Psi (z)=e^{2tz}$, $z\in \mathbb{R}$, and $\varphi(\theta)=\cos \theta -1$, $\theta \in(0,\pi)$. It is clear that $\phi _t(\theta )=(\Psi\circ \varphi)(\theta)$, $\theta \in (0,\pi)$. 

By using Fa\`a di Bruno formula (\cite{FaBruno}) we can write
\begin{align*}
\partial_\theta ^k\phi _t(\theta)&=(\Psi \circ\varphi)^{(k)}(\theta)\\
&=\hspace{-1cm}\sum_{\substack{(m_1,...,m_k)\in\mathbb{N}^n\\m_1+2m_2+...+km_k=k}}\frac{k!}{m_1!1!^{m_1}m_2!2!^{m_2}\cdots m_k!k!^{m_k}}\Psi ^{(m_1+...+m_k)}(\varphi (\theta))\prod_{r=1}^k(\varphi ^{(r)}(\theta))^{m_r},\quad \theta \in (0,\pi).
\end{align*}
Since for each $r\in \mathbb{N}$, $\Psi^{(r)}(z)=(2t)^r\Psi (z)$,  and when $r\geq 1$,
$$
\varphi ^{(r)}(\theta)=\cos (\theta +r\pi/2)=(-1)^{E[(r+1)/2]}h_r(\theta)=\left\{\begin{array}{ll}
(-1)^{(r+1)/2}\sin \theta,&\mbox{ if }r\mbox{ is odd},\\[0.2cm]
(-1)^{r/2}\cos \theta,&\mbox{ if } r\mbox{ is even},
\end{array}
\right.
,\quad \theta \in (0,\pi), 
$$
we can finish the proof of the lemma.
\end{proof}
An immediate consequence of Lemma \ref{Lem2.1} is the following.
\begin{lem}\label{Lem2.2}
Let $k\in\mathbb N$ be odd. Then $\partial_\theta ^k\phi _t(0)=\partial_\theta ^k\phi _t(\pi)=0$, $t\in (0,\infty )$.
\end{lem}

Next property follows by using partial integration and Lemma \ref{Lem2.2}.
\begin{lem}\label{Lem2.3}
Let $k\in\mathbb N$. We have that
$$
G(m,t)=\frac{(-1)^{[(k+1)/2]}}{\pi m^k}\int_0^\pi \partial _\theta ^k\phi _t(\theta)h_k (m\theta )d\theta ,\quad m\in \mathbb{Z}\setminus\{0\},\;t>0.
$$
\end{lem}

In the next three propositions we collect some fundamental boundedness properties that will be useful for our proofs. 
\begin{prop}\label{Prop2.4}
Let $k\in\mathbb N$, $k\geq 1$. Then,
$$
\sup_{t\in (0,\infty )}\int_0^\pi |\partial _\theta ^k\phi _t(\theta)|\theta ^{k-1}d\theta <\infty.
$$
\end{prop}
\begin{proof}
According to Lemma \ref{Lem2.1}
\begin{equation}\label{partialphit}
\partial_\theta ^k\phi _t(\theta)=\sum_{\substack{(m_1,...,m_k)\in\mathbb{N}^k\\m_1+2m_2+...+km_k=k}}a_{m_1,...,m_k}A_{m_1,...,m_k}(t,\theta),\quad t>0,\;\theta \in (0,\pi),
\end{equation}
where, for each $(m_1,...,m_k)\in\mathbb{N}^k$ such that $m_1+2m_2+...+km_k=k$, $a_{m_1,...,m_k}\in \mathbb{R}$ and
$$
A_{m_1,...,m_k}(t,\theta)=e^{-2t(1-\cos \theta)}t^{m_1+...m_k}(\cos \theta)^{\alpha_{m_1,...,m_k}}(\sin \theta)^{\beta_{m_1,...,m_k}},\quad t>0,\;\theta \in (0,\pi),
$$
being
$$
\alpha _{m_1,...,m_k}=\sum_{j=1}^{E[k/2]}m_{2j}\quad \mbox{ and }\quad \beta _{m_1,...,m_k}=\sum_{j=1}^{E[(k+1)/2]}m_{2j-1}.
$$

In order to prove that
$$
\sup_{t\in (0,\infty )}\int_0^\pi |\partial _\theta ^k\phi _t(\theta)|\theta ^{k-1}d\theta <\infty,
$$
it is sufficient to see that, for every $(m_1,...,m_k)\in \mathbb{N}^k$ such that $m_1+2m_2+...+km_k=k$, there exist $C,c>0$ such that

$$
|A_{m_1,...,m_k}(t,\theta)\theta ^{k-1}|\leq Ct\theta e^{-ct\theta^2},\quad t>0,\;\theta \in (0,\pi).
$$
Let $(m_1,...,m_k)\in \mathbb{N}^k$ such that $m_1+2m_2+...+km_k=k$. Since $|\sin \theta| \leq |\theta|$, $\theta \in\mathbb{R}$, and $1-\cos \theta\geq c\theta ^2$, $\theta \in (0,\pi)$, we can write
\begin{align}\label{A}
|A_{m_1,...,m_k}(t,\theta)\theta ^{k-1}|&\leq Ce^{-ct\theta ^2}t^{\sum_{j=1}^km_j}\theta ^{k-1+\beta_{m_1,...,m_k}}\nonumber\\
&\leq Ct\theta e^{-ct\theta ^2}\theta ^{k+\beta_{m_1,...,m_k}-2\sum_{j=1}^km_j}\leq Ct\theta e^{-ct\theta ^2},\quad t>0, \;\theta \in (0,\pi).
\end{align}

In the last inequality we have taken into account that
\begin{equation}\label{2.1}
\sigma_{m_1,...,m_k}:=k+\beta_{m_1,...,m_k}-2\sum_{j=1}^km_j\geq 0.
\end{equation}
Indeed, we can write
\begin{align*}
k&=\sum_{j=1}^{E[k/2]}2jm_{2j}+\sum_{j=1}^{E[(k+1)/2]}(2j-1)m_{2j-1}\geq 2(\alpha _{m_1,..,m_k}+\beta_{m_1,...,m_k})-\beta_{m_1,...,m_k}\\
&=2\sum_{j=1}^km_j-\beta _{m_1,...,m_k}.
\end{align*}
Moreover, $\sigma_{m_1,...,m_k}=0$ if, and only if $m_3=m_4=...=m_k=0$, that is, $k=m_1+2m_2$.

\end{proof}

From Proposition 2.4 it follows immediately that, for every $k\in\mathbb N$, $k\geq 1$, there exists $C>0$ such that
\begin{equation}\label{acot2.4}
\sup_{t\in (0,\infty )}\Big|\int_0^\pi \partial _\theta ^k\phi _t(\theta)\theta ^{k-1}\sin (z\theta)d\theta \Big|\leq C,\quad z\in \mathbb{R}.
\end{equation}

\begin{prop}\label{Prop2.5}
Let $k\in\mathbb N$, $k\geq 1$. Then,
$$
\sup_{z\geq 1}\left\|t\partial_t\int_0^\pi \partial _\theta ^k\phi _t(\theta)\theta ^{k-1}\sin (z\theta)d\theta\right\|_{L^2((0,\infty ),dt/t)}<\infty.
$$
\end{prop}
\begin{proof}
From \eqref{partialphit} we have, for every $t>0$ and $\theta \in (0,\pi)$,
\begin{equation}\label{partialtphi}
t\partial_t\partial_\theta ^k\phi _t(\theta)=\hspace{-0.5cm}\sum_{\substack{(m_1,...,m_k)\in\mathbb{N}^k\\m_1+2m_2+...+km_k=k}}\hspace{-0.5cm}a_{m_1,...,m_k}\Big(-2t(1-\cos \theta )+\sum_{j=1}^km_j\Big)A_{m_1,...,m_k}(t,\theta).
\end{equation}
Thus, in order to establish our result it is sufficient to see that, for every $(m_1,...,m_k)\in \mathbb{N}^k$ such that $m_1+2m_2+...+km_k=k$, we can find $C>0$, such that, for each $z\geq1$,
$$
I_{m_1,...,m_k}(z):=\left\|\int_0^\pi A_{m_1,...,m_k}(t,\theta)\theta ^{k-1}\sin (z\theta)d\theta\right\|_{L^2((0,\infty ),dt/t)}\leq C.
$$
and
$$
J_{m_1,...,m_k}(z):=\left\|\int_0^\pi t(1-\cos \theta)A_{m_1,...,m_k}(t,\theta)\theta ^{k-1}\sin (z\theta)d\theta\right\|_{L^2((0,\infty ),dt/t)}\leq C.
$$

Let $(m_1,...,m_k)\in \mathbb{N}^k$, with $m_1+2m_2+...+km_k=k$ and consider $\sigma_{m_1,...,m_k}$ as in \eqref{2.1}. 

In the case that $\sigma_{m_1,...,m_k}>0$, by using \eqref{A} and Minkowski integral inequality we obtain that
\begin{align*}
I_{m_1,...,m_k}(z)+J_{m_1,...,m_k}(z)&\leq C\int_0^\pi \Big\|t\theta e^{-ct\theta ^2}\theta ^{\sigma_{m_1,...,m_k}}(1+t\theta ^2)\Big\|_{L^2((0,\infty ),dt/t)}d\theta\\
&\leq C\int_0^\pi \theta ^{\sigma_{m_1,...,m_k}+1}\left(\int_0^\infty te^{-ct\theta^2}dt\right)^{1/2}d\theta\\
&=C\int_0^\pi \theta^{\sigma_{m_1,...,m_k}-1}d\theta\leq C,\quad z\in \mathbb{R}.
\end{align*}

Assume now that $\sigma_{m_1,...,m_k}=0$. Then $m_3=m_4=...=m_k=0$, $k=m_1+2m_2$ and
\begin{equation}\label{A0}
A_{m_1,...,m_k}(t,\theta )=e^{-2t(1-\cos \theta)}t^{m_1+m_2}(\cos  \theta)^{m_2}(\sin\theta )^{m_1},\quad t>0,\;\theta \in (0,\pi).
\end{equation}
Note that, since $k\geq 1$, $m_1+m_2\geq 1$. 

For every $z\in \mathbb{R}$ we can write
\begin{align*}
(I_{m_1,...,m_k}(z))^2&=\int_0^\infty \int_0^\pi \int_0^\pi e^{-2t(2-\cos \theta _1-\cos \theta _2)}t^{2(m_1+m_2)-1}(\cos \theta _1\cos\theta _2)^{m_2}(\sin \theta _1\sin \theta _2)^{m_1}\\
&\qquad \qquad \qquad \times (\theta _1\theta _2)^{k-1}\sin (z\theta _1)\sin(z\theta _2)d\theta _1d\theta _2dt\\
&\hspace{-2cm}=\frac{\Gamma (2(m_1+m_2))}{2^{2(m_1+m_2)}}\int_0^\pi \int_0^\pi \frac{(\cos \theta _1\cos\theta _2)^{m_2}(\sin \theta _1\sin \theta _2)^{m_1}(\theta _1\theta _2)^{k-1}\sin (z\theta _1)\sin(z\theta _2)}{(2-\cos \theta _1-\cos \theta _2)^{2(m_1+m_2)}}d\theta _1d\theta _2.
\end{align*}
Let $F$ and $G$ be the functions given by
$$
F(\theta _1,\theta _2)=\frac{(\cos \theta _1\cos\theta _2)^{m_2}(\sin \theta _1\sin \theta _2)^{m_1}}{(2(2-\cos \theta _1-\cos \theta _2))^{2(m_1+m_2)}},\quad \theta _1,\theta _2\in (0,\pi),
$$
and 
$$
H(\theta _1,\theta _2)=\frac{(\theta _1\theta _2)^{m_1}}{(\theta _1^2+\theta _2^2)^{2(m_1+m_2)}},\quad \theta _1,\theta _2\in (0,\pi).
$$
We have that
\begin{align*}
|F(\theta _1,\theta _2)-H(\theta _1,\theta _2)|&\leq \left|\frac{(\cos \theta _1\cos\theta _2)^{m_2}(\sin \theta _2)^{m_1}[(\sin \theta _1)^{m_1}-\theta _1^{m_1}]}{(2(2-\cos \theta _1-\cos \theta _2))^{2(m_1+m_2)}}\right|\\
&\quad +\left|\frac{\theta _1^{m_1}(\cos \theta _1\cos\theta _2)^{m_2}[(\sin \theta _2)^{m_1}-\theta _2^{m_1}]}{(2(2-\cos \theta _1-\cos \theta _2))^{2(m_1+m_2)}}\right|\\
&\quad +\left|\frac{(\theta _1\theta _2)^{m_1}(\cos \theta _1\cos\theta _2)^{m_2}}{(2(2-\cos \theta _1-\cos \theta _2))^{2(m_1+m_2)}}-\frac{(\theta _1\theta _2)^{m_1}(\cos \theta _1\cos\theta _2)^{m_2}}{(\theta _1^2+2(1-\cos \theta _2))^{2(m_1+m_2)}}\right|\\
&\quad +\left|\frac{(\theta _1\theta _2)^{m_1}(\cos \theta _1\cos\theta _2)^{m_2}}{(\theta _1^2+2(1-\cos \theta _2))^{2(m_1+m_2)}}-\frac{(\theta _1\theta _2)^{m_1}(\cos \theta _1\cos\theta _2)^{m_2}}{(\theta _1^2+\theta _2^2)^{2(m_1+m_2)}}\right|\\
&\quad +\left|\frac{(\theta _1\theta _2)^{m_1}(\cos \theta _1\cos\theta _2)^{m_2}}{(\theta _1^2+\theta _2^2)^{2(m_1+m_2)}}-\frac{(\theta _1\theta _2)^{m_1}(\cos \theta _2)^{m_2}}{(\theta _1^2+\theta _2^2)^{2(m_1+m_2)}}\right|\\
&\quad +\left|\frac{(\theta _1\theta _2)^{m_1}(\cos\theta _2)^{m_2}}{(\theta _1^2+\theta _2^2)^{2(m_1+m_2)}}-\frac{(\theta _1\theta _2)^{m_1}}{(\theta _1^2+\theta _2^2)^{2(m_1+m_2)}}\right|\\
&:=\sum_{j=1}^6R_j(\theta _1,\theta _2),\quad \theta _1,\theta _2\in (0,\pi).
\end{align*}
We observe that if $m_1=0$, then $R_1(\theta_1,\theta_2)=R_2(\theta_1,\theta_2)=0$ and in the case that $m_2=0$, we have that $R_5(\theta _1,\theta_2)=R_6(\theta _1,\theta_2)=0$.  By taking into account that $|\sin \theta -\theta |\leq C\theta ^3$ and $|1-\cos \theta-\theta ^2/2|\leq c\theta ^4$, $\theta \in (0,\pi)$, and applying the mean value theorem we get, for each $\theta \in (0,\pi)$,
$$
|(\sin \theta )^{m_1}-\theta ^{m_1}|\leq C\theta ^{m_1+2},
$$
$$
|1-(\cos \theta)^{m_2}|\leq C\theta ^2,
$$
and, if $\alpha \in \mathbb{R}$,
$$
\left|\frac{1}{(2(1-\cos \theta)+\alpha)^{2(m_1+m_2)}}-\frac{1}{(\theta ^2+\alpha )^{2(m_1+m_2)}}\right|\leq C\frac{\theta ^4}{(\theta ^2+\alpha )^{2(m_1+m_2)+1}}\leq C\frac{\theta ^2}{(\theta ^2+\alpha )^{2(m_1+m_2)}}.
$$
By considering these estimates and that $0\leq \sin \theta \leq \theta$ and $|1-\cos \theta |\geq C\theta ^2$, $\theta \in (0,\pi)$, we obtain that
\begin{align*}
|F(\theta_1,\theta_2)-H(\theta _1,\theta_2)|&\leq C\frac{(\theta _1\theta_2)^{m_1}(\theta _1^2+\theta_2^2)}{(\theta_1^2+\theta _2^2)^{2(m_1+m_2)}}=C\frac{(\theta _1\theta_2)^{m_1}}{(\theta_1^2+\theta _2^2)^{2(m_1+m_2)-1}}\\
&\leq \frac{C}{(\theta_1^2+\theta_2^2)^{m_1+2m_2-1}}=\frac{C}{(\theta_1^2+\theta_2^2)^{k-1}},\quad \theta _1,\theta_2\in (0,\pi).
\end{align*}
Then, we get
\begin{align*}
(I_{m_1,...,m_k}(z))^2&\leq C\left(\int_0^\pi\int_0^\pi |F(\theta_1,\theta_2)-H(\theta _1,\theta_2)|(\theta _1\theta _2)^{k-1}d\theta _1d\theta _2\right.\\
&\quad +\left.\left|\int_0^\pi\int_0^\pi H(\theta_1,\theta_2)(\theta_1\theta_2)^{k-1}\sin (z\theta_1)\sin(z\theta_2)d\theta_1d\theta_2\right|\right)\\
&\leq C\left(1+\left|\int_0^{\pi}\int_0^{\pi} \frac{(\theta_1\theta_2)^{m_1+k-1}}{(\theta_1^2+\theta_2^2)^{2(m_1+m_2)}}\sin (z\theta_1)\sin(z\theta_2)d\theta_1d\theta_2\right|\right),\quad z\in \mathbb{R}.
\end{align*}
To analyze the last integral we proceed as follows:
\begin{align*}
\left|\int_0^{\pi}\int_0^{\pi} \frac{(\theta_1\theta_2)^{m_1+k-1}}{(\theta_1^2+\theta_2^2)^{2(m_1+m_2)}}\sin (z\theta_1)\sin(z\theta_2)d\theta_1d\theta_2\right|&\\
&\hspace{-7cm}=\left|\int_0^{z\pi}\int_0^{z\pi} \frac{(v_1v_2)^{2(m_1+m_2)-1}}{(v_1^2+v_2^2)^{2(m_1+m_2)}}\sin v_1\sin v_2dv_1dv_2\right|\\
&\hspace{-7cm}\leq C\left(\int_0^1\int_0^1dv_1dv_2+\left|\int_0^{z\pi}\int_1^{z\pi} \frac{(v_1v_2)^{2(m_1+m_2)-1}}{(v_1^2+v_2^2)^{2(m_1+m_2)}}\sin v_1\sin v_2dv_1dv_2\right|\right),\quad z\in \mathbb{R}.
\end{align*}
By using partial integration we get
\begin{align*}
I(z):=\int_0^{z\pi}\int_1^{z\pi} \frac{(v_1v_2)^{2(m_1+m_2)-1}}{(v_1^2+v_2^2)^{2(m_1+m_2)}}\sin v_1\sin v_2dv_1dv_2&\\
&\hspace{-6cm}=\int_0^{z\pi}\left(\frac{\cos 1}{(1+v_2^2)^{2(m_1+m_2)}}-\frac{\cos (z\pi)(z\pi)^{2(m_1+m_2)-1}}{((z\pi)^2+v_2^2)^{2(m_1+m_2)}}\right)v_2^{2(m_1+m_2)-1}\sin v_2dv_2\\
&\hspace{-6cm}\quad +\int_0^{z\pi}\int_1^{z\pi}\frac{v_1^{2(m_1+m_2)-2}[(2(m_1+m_2)-1)v_2^2-(2(m_1+m_2)+1)v_1^2]}{(v_1^2+v_2^2)^{2(m_1+m_2)+1}}\\
&\hspace{-4cm} \times v_2^{2(m_1+m_2)-1}\sin v_2\cos v_1dv_1dv_2,,\quad z\in \mathbb{R}.
\end{align*}
Then, it follows that
\begin{align*}
|I(z)|&\leq C\left(\int_0^\infty \left(\frac{1}{(1+v_2)^{2(m_1+m_2)+1}}+\frac{1}{z^2+v_2^2}\right)dv_2+\int_0^\infty \int_1^\infty \frac{v_1^{2(m_1+m_2)-2}v_2^{2(m_1+m_2)-1}}{(v_1^2+v_2^2)^{2(m_1+m_2)}}dv_1dv_2\right)\\
&\leq C\left(1+\frac{1}{z}+\int_0^\infty \frac{v_2^{2(m_1+m_2)-1}}{(1+v_2)^{2(m_1+m_2)+1}}dv_2\right)\leq C\left(1+\frac{1}{z}\right)\leq C,\quad z\geq 1.
\end{align*}
By combining the above estimates we conclude that $I_{m_1,...,m_k}(z)\leq C$, when $z\geq 1$.

To deal with $J_{m_1,...,m_k}(z)$, $z\in \mathbb{R}$, we proceed as in the $I_{m_1,...,m_k}$-case, by taking into account that
\begin{align*}
(J_{m_1,...,m_k}(z))^2&=\frac{\Gamma (2(m_1+m_2)+2)}{2^{2(m_1+m_2)+2}}\int_0^\pi \int_0^\pi \frac{(1-\cos\theta_1)(1-\cos \theta _2)}{(2-\cos \theta_1-\cos\theta_2)^2}\\
&\quad \times \frac{(\cos \theta _1\cos\theta _2)^{m_2}(\sin \theta _1\sin \theta _2)^{m_1}(\theta _1\theta _2)^{k-1}\sin (z\theta _1)\sin(z\theta _2)}{(2-\cos \theta _1-\cos \theta _2)^{2(m_1+m_2)}}d\theta _1d\theta _2,\quad z\in \mathbb{R},
\end{align*}
and considering the functions $\widetilde{F}$ and $\widetilde{H}$ given by
$$
\widetilde{F}(\theta_1,\theta_2)=\frac{(1-\cos\theta_1)(1-\cos \theta _2)}{(2-\cos \theta_1-\cos\theta_2)^2}F(\theta_1,\theta_2),\quad \theta _1,\theta_2\in (0,\pi),
$$
and
$$
\widetilde{H}(\theta_1,\theta_2)=\frac{(\theta_1\theta_2)^2}{(\theta _1^2+\theta_2^2)^2}H(\theta_1,\theta_2),\quad \theta_1,\theta_2\in (0,\pi).
$$
in the role of $F$ and $H$, respectively. 
\end{proof}

\begin{prop}\label{Prop2.7}
Let $k\in\mathbb N$, $k\geq 1$ and $\Psi \in L^\infty (0,\infty)$. Then,
$$
\sup_{z\geq 1}\left|\int_0^\infty \Psi(t)\int_0^\pi \partial _t\partial _\theta ^k\phi _t(\theta)\theta ^{k-1}\sin (z\theta)d\theta dt\right|<\infty.
$$
\end{prop}
\begin{proof}
According to \eqref{partialtphi} it is sufficient to show that, for every $(m_1,...,m_k)\in \mathbb{N}^k$ such that $m_1+2m_2+...+km_k=k$, there exists $C>0$, such that, for each $z\geq1$,
$$
R_{m_1,...,m_k}(z):=\left|\int_0^\infty \Psi (t)\int_0^\pi A_{m_1,...,m_k}(t,\theta)\theta ^{k-1}\sin (z\theta)d\theta\frac{dt}{t}\right|\leq C.
$$
and
$$
S_{m_1,...,m_k}(z):=\left|\int_0^\infty \Psi (t)\int_0^\pi (1-\cos \theta)A_{m_1,...,m_k}(t,\theta)\theta ^{k-1}\sin (z\theta)d\theta dt\right|\leq C.
$$

Let $(m_1,...,m_k)\in \mathbb{N}^k$, with $m_1+2m_2+...+km_k=k$ and let $\sigma_{m_1,...,m_k}$ be as in \eqref{2.1}. When $\sigma_{m_1,...,m_k}>0$, we can use the estimate \eqref{A} to get
\begin{align*}
R_{m_1,...,m_k}(z)+S_{m_1,...,m_k}(z)&\leq C\|\Psi \|_\infty \int_0^\pi \int_0^\infty \theta ^{\sigma_{m_1,...,m_k}+1}e^{-ct\theta ^2}(1+t\theta ^2)dtd\theta\\
&\leq C\|\Psi \|_\infty\int_0^\pi \theta ^{\sigma_{m_1,...,m_k}-1}d\theta \leq C,\quad z\in \mathbb{R}. 
\end{align*}

Suppose now that $\sigma_{m_1,...,m_k}=0$ (recall that this means that $k=m_1+2m_2$) and consider the function 
$$
H(t,\theta):=e^{-t\theta ^2}t^{m_1+m_2}\theta^{m_1}, \quad t\in (0,\infty ),\;\theta \in (0,\pi). 
$$
By taking into account \eqref{A0}, that $|\sin \theta -\theta|\leq C\theta ^3$, $|2(1-\cos \theta) -\theta ^2|\leq C\theta ^4$, and $c\theta ^2\leq 1-\cos \theta \leq C\theta ^2$, for $\theta \in (0,\pi)$, and by using the mean value theorem we obtain that
\begin{align*}
|A_{m_1,...,m_k}(t,\theta)-H(t,\theta)|&\leq t^{m_1+m_2}\Big(|e^{-2t(1-\cos \theta)}(\cos \theta )^{m_2}[(\sin \theta )^{m_1}-\theta ^{m_1}]|\\
&\hspace{-2cm}\quad +|\theta ^{m_1}(\cos \theta )^{m_2}[e^{-2t(1-\cos \theta)}-e^{-t\theta ^2}]|+|\theta ^{m_1}e^{-t\theta ^2}[1-(\cos \theta )^{m_2}]|\Big)\\
&\hspace{-2cm}\leq Ct^{m_1+m_2}e^{-ct\theta ^2}\theta ^{m_1}(\theta ^2+t\theta ^4)\leq Ct^{m_1+m_2}e^{-ct\theta ^2}\theta ^{m_1+2},\quad t>0,\;\theta \in (0,\pi).
\end{align*}

We deduce that
\begin{align*}
R_{m_1,...,m_k}(z)&\\
&\hspace{-1cm}\leq C\|\Psi \|_{L^\infty(0,\infty)} \left(\int_0^\infty \int_0^\pi t^{m_1+m_2-1}e^{-ct\theta ^2}\theta ^{m_1+k+1}d\theta dt+\int_0^\infty \left|\int_0^\pi H(t,\theta )\theta ^{k-1}\sin (z\theta)d\theta\right|\frac{dt}{t}\right)\\
&\hspace{-1cm}\leq C\|\Psi \|_{L^\infty(0,\infty)} \left(\int_0^\pi \theta d\theta+z^{-2(m_1+m_2)}\int_0^\infty t^{m_1+m_2-1}\left|\int_0^{z\pi}e^{-tv^2/z^2}v^{2(m_1+m_2)-1}\sin vdv\right|dt\right)\\
&\hspace{-1cm}\leq C\|\Psi \|_{L^\infty(0,\infty)} \left(1+z^{-2(m_1+m_2)}\int_0^\infty t^{m_1+m_2-1}\int_0^1e^{-tv^2/z^2}v^{2(m_1+m_2)}dvdt\right.\\
&\left.+z^{-2(m_1+m_2)}\int_0^\infty t^{m_1+m_2-1}\left|\int_1^{z\pi}e^{-tv^2/z^2}v^{2(m_1+m_2)-1}\sin vdv\right|dt\right)\\
&\hspace{-1cm}\leq C\|\Psi \|_{L^\infty(0,\infty)} \left(1+z^{-2(m_1+m_2)}\int_0^\infty t^{m_1+m_2-1}\left|\int_1^{z\pi}e^{-tv^2/z^2}v^{2(m_1+m_2)-1}\sin vdv\right|dt\right)\\
&\hspace{-1cm}=C\|\Psi \|_{L^\infty(0,\infty)} (1+I(z)),\quad z\in (0,\infty ).
\end{align*}
By making use of integration by parts we obtain that
\begin{align*}
I(z)&\leq z^{-2(m_1+m_2)}\left(\int_0^\infty t^{m_1+m_2-1}\left|e^{-t\pi^2}(z\pi)^{2(m_1+m_2)-1}\cos(z\pi)-e^{-t/z^2}\cos 1\right|dt\right.\\
&\quad \left.+\int_0^\infty t^{m_1+m_2-1}\left| \int_1^{z\pi}e^{-tv^2/z^2}v^{2(m_1+m_2)-2}\Big(2(m_1+m_2)-1-2\frac{tv^2}{z^2}\Big)\cos vdv\right|dt\right)\\
&\leq C\left(\frac{1}{z}+1+z^{-2(m_1+m_2)}\int_0^\infty \int_1^\infty t^{m_1+m_2-1}e^{-ctv^2/z^2}v^{2(m_1+m_2)-2}dvdt\right)\\
&\leq C\left(\frac{1}{z}+1+\int_1^\infty \frac{dv}{v^2}\right)\leq C\left(\frac{1}{z}+1\right),\quad z\in (0,\infty ).
\end{align*}
By combining the above estimates we conclude that $R_{m_1,...,m_k}(z)\leq C$, $z\geq 1$. The boundedness property for $S_{m_1,...,m_k}$ can be established by proceeding in a similar way. 
\end{proof}

We will use the following generalizations of \eqref{acot2.4} and Propositions \ref{Prop2.5} and \ref{Prop2.7} in the proofs of our theorems. 

\begin{prop}\label{Prop2.6}
Let $n\in\mathbb N$. For every $k\in \mathbb{N}$, $k\geq \max\{n,1\}$, there exists $C>0$ such that:
\begin{equation}\label{hna}
\sup_{t\in (0,\infty )}\Big|\int_0^\pi \partial _\theta ^k\phi _t(\theta)\theta ^{k-n}h_n (z\theta)d\theta \Big|\leq Cz^{n-1},\quad z\in (0,\infty),
\end{equation}
\begin{equation}\label{hnb}
\left\|t\partial_t\int_0^\pi \partial _\theta ^k\phi _t(\theta)\theta ^{k-n}h_n (z\theta)d\theta\right\|_{L^2((0,\infty ),dt/t)}\leq Cz^{n-1},\quad z\geq 1,
\end{equation}
and,  for every $\Psi \in L^\infty(0,\infty )$, 
\begin{equation}\label{hnc}
\left|\int_0^\infty \Psi(t)\int_0^\pi \partial _t\partial _\theta ^k\phi _t(\theta)\theta ^{k-n}h_n(z\theta)d\theta dt\right|\leq Cz^{n-1},\quad z\geq 1.
\end{equation}

\end{prop}
\begin{proof}
We proceed by induction. Note that the cases for $n=1$ are the ones in \eqref{acot2.4} and Propositions \ref{Prop2.4}, \ref{Prop2.5} and \ref{Prop2.7} . Let us show the property \eqref{hna} for $n=0$.
Let $k\in \mathbb{N}$, $k\geq 1$. Partial integration leads to
\begin{align}\label{induction0}
\int_0^\pi \partial _\theta ^k\phi _t(\theta)\theta ^k\cos(z\theta)d\theta&=\pi ^k\partial _\theta ^k\phi _t(\pi)\frac{\sin (\pi z)}{z}-\frac{1}{z}\int_0^\pi [\theta ^k\partial _\theta ^{k+1}\phi _t(\theta )+k\theta ^{k-1}\partial _\theta ^k\phi _t(\theta)]\sin (z\theta)d\theta\nonumber\\
&\hspace{-2cm}=p(t)e^{-4t}\frac{\sin (\pi z)}{z}-\frac{1}{z}\int_0^\pi [\theta^k\partial _\theta ^{k+1}\phi _t(\theta)+k\theta ^{k-1}\partial _\theta ^k\phi _t(\theta )]\sin (z\theta )d\theta,\quad z,t\in (0,\infty ),
\end{align}
where $p$ is a polynomial. According to \eqref{acot2.4} we get
$$
\sup_{t\in (0,\infty )}\left|\int_0^\pi \partial _\theta ^k\phi _t(\theta)\theta ^k\cos(z\theta)d\theta\right|\leq \frac{C}{z},\quad z\geq 1.
$$

Suppose now that $n\geq 2$ and, for every $k\geq n-1$, \eqref{hna} holds, when $n$ is replaced by $n-1$. Let $k\geq n$. We observe that
\begin{align*}
\int\partial _\theta ^k\phi _t(\theta)\theta ^{k-n}d\theta&=\sum_{i=0}^{k-n}a_{i,n}\theta ^{k-n-i}\partial_\theta ^{k-i-1}\phi _t(\theta),\quad t>0,\;\theta \in (0,\pi),
\end{align*}
where $a_{i,n}\in \mathbb{R}$, $i=0,...,k-n$, and by partial integration we get
\begin{align*}
\int_0^\pi \partial _\theta ^k\phi _t(\theta)\theta ^{k-n}h_n (z\theta)d\theta&=\sum_{i=0}^{k-n}a_{i,n}\left(\left.\theta ^{k-n-i}\partial _\theta ^{k-i-1}\phi _t(\theta)h_n(z\theta)\right]_{\theta =0}^{\theta=\pi} \right.\\
&\left.\quad +(-1)^nz\int_0^\pi \theta ^{k-n-i}\partial _\theta ^{k-i-1}\phi _t(\theta)h_{n-1}(z\theta)d\theta\right),\quad z,t\in (0,\infty ).
\end{align*}
For each $i=0,...,k-n-1$, it is clear that $\theta ^{k-n-i}\partial _\theta ^{k-i-1}\phi _t(\theta)h_n(z\theta)_{|\theta =0}=0$, and, since $h_n(0)=0$, when $n$ is odd and by Lemma \ref{Lem2.2}, $\partial_\theta ^{n-1}\phi_t(0)=0$, when $n$ is even, $\partial _\theta ^{n-1}\phi _t(0)h_n(0)=0$. Hence
\begin{align}\label{inductionn}
\int_0^\pi \partial _\theta ^k\phi _t(\theta)\theta ^{k-n}h_n (z\theta)d\theta&=\sum_{i=0}^{k-n}a_{i,n}\left(\pi ^{k-n-i}\partial _\theta ^{k-i-1}\phi _t(\pi)h_n(z\pi) \right.\nonumber\\
&\left.\quad +(-1)^nz\int_0^\pi \theta ^{k-n-i}\partial _\theta ^{k-i-1}\phi _t(\theta)h_{n-1}(z\theta)d\theta\right),\quad z,t\in (0,\infty ).
\end{align}

As it was mentioned before, for every $\ell \in \mathbb{N}$ there exists a polynomial $p_\ell $ such that
$$
\partial_\theta ^\ell \phi _t(\pi)=p_\ell (t)e^{-4t},\quad t>0.
$$
Then, for a certain polynomial $q_n$, we obtain
\begin{align*}
\int_0^\pi \partial _\theta ^k\phi _t(\theta)\theta ^{k-n}h_n (z\theta)d\theta&=q_n(t)e^{-4t}h_n(z\pi )\\
&\quad +(-1)^nz\sum_{i=0}^{k-n}a_{i,n}\int_0^\pi \theta ^{k-n-i}\partial _\theta ^{k-i-1}\phi _t(\theta)h_{n-1}(z\theta)d\theta,\quad z,t\in (0,\infty ),
\end{align*}
and hence, by using the induction hypothesis we conclude that  
\begin{align*}
\sup_{t\in (0,\infty )}\left|\int_0^\pi \partial _\theta ^k\phi _t(\theta)\theta ^{k-n}h_n (z\theta)d\theta\right|&\leq C\left(\sup_{t\in (0,\infty )}|q_n(t)e^{-4t}|\right.\\
&\hspace{-3cm}\quad +\left. z\sum_{i=0}^{k-n}\sup_{t\in (0,\infty )}\left|\int_0^\pi \theta ^{k-n-i}\partial _\theta ^{k-i-1}\phi _t(\theta)h_{n-1}(z\theta)d\theta\right|\right)\\
&\hspace{-3cm} \leq C(1+z^{n-1})\leq Cz^{n-1},\quad z\geq 1.
\end{align*}

From Propositions \ref{Prop2.5} and \ref{Prop2.7}, by taking into account \eqref{induction0} and \eqref{inductionn} and proceeding by induction as before we obtain \eqref{hnb} and \eqref{hnc}. 
\end{proof}

\begin{prop}\label{derivadaHk}
Let $k\in \mathbb{N}$ and consider the function
$$
H_k(z,t):=\frac{1}{z^k}\int_0^\pi\partial_\theta^k\phi_t(\theta)h_k( z\theta)d\theta,\quad z,t\in(0,\infty).
$$
We have that
\begin{equation}\label{eq3.3}
\sup_{t>0}|\partial_z^kH_k(z,t)|\leq\frac{C}{z^{k+1}},\quad z\in (0,\infty),
\end{equation}

\begin{equation}\label{eq4.3}
\|t\partial_t\partial_z^kH_k(z,t)\|_{L^2((0,\infty),\frac{dt}{t})}\leq\frac{C}{z^{k+1}},\quad z\geq 1,
\end{equation}
and, for every $\Psi \in L^\infty(0,\infty )$, 
\begin{equation}\label{eq4.4}
\left|\int_0^\infty\Psi(t)\partial _t\partial_z^kH_k(z,t)dt\right|\leq\frac{C}{z^{k+1}},\quad z\geq 1.
\end{equation}
\end{prop}
\begin{proof}
We can write
$$
\partial_z ^kH_k(z,t)=\sum_{j=0}^k\frac{c_{k,j}}{z^{k+j}}\int_0^\pi\partial_\theta^k\phi_t(\theta)\theta^{k-j}(\partial_\theta ^{k-j}h_k)(z\theta)d\theta,\;\;z,t\in(0,\infty),$$
for certain $c_{k,j}\in\mathbb R$, $j=0,...,k$.

Now we observe that, for every $r\in \mathbb{N}$,
$$
\partial_\theta ^rh_k(\theta)=(-1)^{E[(r+1)/2]}\left\{\begin{array}{ll}
	h_k(\theta),&\mbox{ if }r\mbox{ is even},\\
    (-1)^{k}h_{k+1}(\theta),&\mbox{ if }r\mbox{ is odd},
\end{array}\right.\;\;\;\theta\in(0,\pi).
$$
Then, for each $j=0,...,k$, if $k-j$ is even, that is, $k,j$ are both odd or even, $\partial _\theta ^{k-j}h_k=(-1)^{(k-j)/2}h_k=(-1)^{(k-j)/2}h_j$. In the case that $k-j$ is odd we have that
$\partial _\theta ^{k-j}h_k=(-1)^{(k-j+1)/2+k}h_{k+1}=(-1)^{(k-j+1)/2)+k}h_j$.

We deduce that 
$$
\partial_z ^kH_k(z,t)=\sum_{j=0}^k\frac{\widetilde{c}_{k,j}}{z^{k+j}}\int_0^\pi \partial _\theta ^k\phi _t(\theta) \theta ^{k-j}h_j(\theta)d\theta,\quad z,t\in (0,\infty ),
$$
for certain $\widetilde{c}_{k,j}\in \mathbb{R}$, $j=0,...,k$, and Proposition \ref{Prop2.6} allows us to conclude our result.  
\end{proof}

\section{Proof of Theorem \ref{Th1.2} for the maximal operator}\label{Section3}

We recall that the maximal operator $W_*$ associated to the semigroup $\{W_t\}_{t>0}$ is defined by
$$W_*(f)=\|W_t(f)\|_{L^\infty(0,\infty)},$$
for every $f\in\ell^r(\mathbb Z)$, $1\leq r\leq\infty$.

In \cite{CGRTV} the behavior of $W_*$ in weighted $\ell^p$-spaces is studied by considering $W_*$ as a $L^\infty$-valued singular integral and by using vectorial Calder\'on-Zygmund theory (\cite{RRT}). In \cite[Proposition 3]{CGRTV} it was proved that, for certain $C>0$,
\begin{equation}\label{eq3.1}
\|G(n,\cdot)\|_{L^\infty(0,\infty)}\leq\frac{C}{|n|+1},\quad n\in\mathbb Z,
\end{equation}
and
\begin{equation}\label{eq3.2}
\|G(n+1,\cdot)-G(n,\cdot)\|_{L^\infty(0,\infty)}\leq\frac{C}{|n|^2+1},\quad n\in\mathbb Z.
\end{equation}
By using (\ref{eq3.1}) and (\ref{eq3.2}) a standard procedure (see, for instance, \cite[p. 586]{CW}) allows to prove that, if $1/2<p\leq 1$ there exists $C>0$ such that
$$\|W_*(\mathfrak b)\|_p\leq C,$$
for every $(p,2)$-atom $\mathfrak b$. Then, if $1/2<p\leq 1$, since $\mathcal H^p(\mathbb Z)$ is continuously contained in $\ell^2(\mathbb Z)$ and $W_*$ is bounded from $\ell^2(\mathbb Z)$ into itself, $W_*$ define a bounded operator from $\mathcal H^p(\mathbb Z)$ into $\ell^p(\mathbb Z)$ (Theorem \ref{Th1.1}).

Our objective is to see that the this last property can be extended to those $p\in(0,1/2]$. In order to do this we are going to improve the estimate (\ref{eq3.2}).

For every $k\in\mathbb N$ we consider the function
$$
F_k(z,t)=\frac{(-1)^{E[(k+1)/2]}}{\pi}H_k(z,t),\quad z,t\in(0,\infty),
$$
where $H_k$ is the function defined in Proposition \ref{derivadaHk}. According to Lemma \ref{Lem2.3}, for each $k\in \mathbb{N}$, $F_k(n,t)=G(n,t)$, $n\in\mathbb N$, $n\geq 1$, and $t\in(0,\infty)$. 

Let $0<p\leq 1$ and $k=E[1/p]$. We take a $(H,p,2)$- atom $\mathfrak b$ such that its support is contained in $B_{\mathbb Z}(n_0,r_0)$ and $\|\mathfrak b\|_2\leq \mu\left(B_{\mathbb Z}(n_0,r_0)\right)^{1/2-1/p}$, for certain $n_0\in\mathbb Z$ and $r_0\geq 1$. We write
$$
\|W_*(\mathfrak b)\|^p_p=\sum_{n\in B_{\mathbb Z}(n_0,2r_0)}|W_*(\mathfrak b)(n)|^p+\sum_{n\notin B_{\mathbb Z}(n_0,2r_0)}|W_*(\mathfrak b)(n)|^p.
$$

Since $W_*$ is bounded from $\ell^2(\mathbb Z)$ into itself we get
\begin{align*}
\sum_{n\in B_{\mathbb Z}(n_0,2r_0)}|W_*(\mathfrak b)(n)|^p&\leq \left(\sum_{n\in B_{\mathbb Z}(n_0,2r_0)}|W_*(\mathfrak b)(n)|^2\right)^{p/2}\mu\left(B_{\mathbb Z}(n_0,2r_0)\right)^{1-p/2}\\
&\leq C\|\mathfrak b\|_2^p \mu\left(B_{\mathbb Z}(n_0,2r_0)\right)^{1-p/2}
\leq C.
\end{align*}
On the other hand, if $n\in\mathbb \mathbb{Z}$ and $n\geq n_0+2r_0$, we have that
\begin{align*}
W_t(\mathfrak b)(n) & =\sum_{m\in B_{\mathbb Z}(n_0,r_0)}G(n-m,t)\mathfrak b(m)=\sum_{m\in B_{\mathbb Z}(n_0,r_0)}F_k(n-m,t)b(m) \\
& = \sum_{m\in B_{\mathbb Z}(n_0,r_0)}\left(F_k(n-m,t)-\sum_{j=0}^{k-1}\frac{(\partial_z^jF_k)(n-n_0,t)}{j!}(n_0-m)^j\right)\mathfrak b(m) \\
& = \sum_{m\in B_{\mathbb Z}(n_0,r_0)}\mathfrak b(m) \int_{n-n_0}^{n-m}\int_{n-n_0}^{x_1}\cdot ...\cdot\int_{n-n_0}^{x_{k-1}}(\partial_z^kF_k)(z,t)dzdx_{k-1}...dx_1,\;\;\;t\in(0,\infty).
\end{align*}
Then from (\ref{eq3.3}) we deduce that 
\begin{align*}
|W_*(\mathfrak b)(n)|&\leq C\sum_{m\in B_{\mathbb Z}(n_0,r_0)}|\mathfrak b(m)| \left|\int_{n-n_0}^{n-m}\int_{n-n_0}^{x_1}\cdot ...\cdot\int_{n-n_0}^{x_{k-1}}\frac{1}{|z|^{k+1}}dzdx_{k-1}...dx_1\right|\\
 & \leq C\sum_{m\in B_{\mathbb Z}(n_0,r_0)}|\mathfrak b(m)| \frac{|m-n_0|^k}{|n-n_0|^{k+1}} \leq C \frac{\|\mathfrak b\|_2}{|n-n_0|^{k+1}}\left(\sum_{m\in B_{\mathbb Z}(n_0,r_0)} |m-n_0|^{2k}\right)^{1/2}\\
  & \leq C\frac{\|b\|_2}{|n-n_0|^{k+1}}r_0^{k+1/2}\leq C\frac{r_0^{k+1-1/p}}{|n-n_0|^{k+1}}, \quad n\in \mathbb{Z},\;n\geq n_0+2r_0. 
\end{align*}
If $n\in\mathbb{Z}$ and $n\leq n_0-2r_0$, since $G(\ell,t)=G(-\ell, t)$, $\ell\in\mathbb Z$ and $t\in(0,\infty)$, we get
\begin{align*}
W_t(\mathfrak b)(n) & = \sum_{m\in B_{\mathbb Z}(n_0,r_0)}G(m-n,t)\mathfrak b(m) \\
& = \sum_{m\in B_{\mathbb Z}(n_0,r_0)}\mathfrak b(m) \int_{n_0-n}^{m-n}\int_{n_0-n}^{x_1}\cdot ...\cdot\int_{n_0-n}^{x_{k-1}}(\partial_z^kF_k)(z,t)dzdx_{k-1}...dx_1,\quad t\in(0,\infty),
\end{align*}
and proceeding as before we obtain
$$
|W_*(\mathfrak b)(n)|\leq C\frac{r_0^{k+1-1/p}}{|n-n_0|^{k+1}}, \quad n\in \mathbb{Z},\;n\leq n_0-2r_0.
$$
We conclude that
$$
\|W_*(\mathfrak b)\|_p^p\leq \sum_{n\notin B_{\mathbb Z}(n_0,2r_0)}|W_*(\mathfrak b)(n)|^p\leq Cr_0^{(k+1)p-1}\sum_{n\notin B_{\mathbb Z}(n_0,2r_0)}\frac{1}{|n-n_0|^{(k+1)p}}\leq C,
$$
because $(k+1)p>1$, where $C>0$ is a constant independent of $\mathfrak{b}$. Then, it follows that the operator $W_*$ is bounded from $\mathcal H^p(\mathbb Z)$ into $\ell^p(\mathbb Z)$.

\section{Proof of Theorem \ref{Th1.2} for the Littlewood-Paley function}\label{Section4}

The vertical Littlewood-Paley function associated to the semigroup $\{W_t\}_{t>0}$ of operators defined by \eqref{funciong} can be expressed as
$$
g(f)(n)=\Big\|t\partial_tW_t(f)(n)\Big\|_{L^2((0,\infty),\frac{dt}{t})}, \quad n\in\mathbb Z,
$$
for every $f\in\ell^p(\mathbb Z)$ and $1\leq p<\infty$.

In \cite{CGRTV} it was proved that the operator $g$ can be seen as a $L^2((0,\infty),dt/t)$- Calder\'on-Zygmund singular integral of convolution type whose kernel is given by
$$\mathfrak k(n,t)=t\partial_tG(n,t),\quad n\in\mathbb Z\;\mbox{and}\;t\in(0,\infty).$$
In \cite[Proposition 4]{CGRTV} it was established that
\begin{equation}\label{eq4.1}
\|\mathfrak k(n,\cdot)\|_{L^2((0,\infty),\frac{dt}{t})}\leq\frac{C}{|n|+1},\;\;\;\;n\in\mathbb Z,
\end{equation}
and
\begin{equation}\label{eq4.2}
\|\mathfrak k(n+1,\cdot)-\mathfrak k(n,\cdot)\|_{L^2((0,\infty),\frac{dt}{t})}\leq\frac{C}{|n|^2+1},\;\;\;\;n\in\mathbb Z.
\end{equation}
A standard procedure allows us to see that $g$ is bounded from $\mathcal H^p(\mathbb Z)$ into $\ell^p(\mathbb Z,\omega)$ when $1/2<p\leq 1$ (Theorem \ref{Th1.1}). In order to prove that $g$ is bounded from $\mathcal H^p(\mathbb Z)$ into $\ell^p(\mathbb Z,\omega)$ when $0<p\leq 1/2$ we need to improve the property (\ref{eq4.2}).

As in Section \ref{Section3}, for every $(H,p,2)$- atom $\mathfrak b$ with support in $B_{\mathbb Z}(n_0,r_0)$ we write
$$
\|g(\mathfrak b)\|^p_p=\sum_{n\in B_{\mathbb Z}(n_0,2r_0)}|g(\mathfrak b)(n)|^p+\sum_{n\notin B_{\mathbb Z}(n_0,2r_0)}|g(\mathfrak b)(n)|^p.
$$
Following the argument developed there, by taking into account the $\ell ^2$-boundedness of the operator $g$ and estimation (\ref{eq4.3}), we can establish that the operator $g$ is bounded from $\mathcal H^p(\mathbb Z)$ into $\ell^p(\mathbb Z)$, for every $0<p\leq 1$,

\section{Proof of Theorem \ref{Th1.3}}

We observe from \eqref{F2} that
$$I_n(t)=\frac{1}{2\pi}\int_{-\pi}^\pi e^{t\cos\theta}e^{-in\theta}d\theta =\mathcal F_{\mathbb{Z}}^{-1}(e^{t\cos\theta})(n),\quad n\in\mathbb Z\mbox{ and }t>0.$$
Then,
\begin{equation}\label{eq5.1}
G(n,t)=\mathcal F_{\mathbb{Z}}^{-1}(e^{-2t(1-\cos \theta)})(n),\quad n\in\mathbb Z\mbox{ and }t>0.
\end{equation}

First we will establish that $T_{\mathfrak m}$ is a Calder\'on-Zygmund singular integral having $(\mathbb Z,|\cdot|,\mu)$ as homogeneous type underlying space.

Let $f\in \ell^2(\mathbb Z)$. We can write, for every $n\in\mathbb Z$,
\begin{align*}
\int_{-\pi}^\pi |\mathfrak{m}(2(1-\cos \theta))\mathcal F_{\mathbb{Z}}(f)(\theta)|d\theta &\leq 2\int_{-\pi}^\pi |\mathcal F_{\mathbb{Z}}(f)(\theta)|(1- \cos \theta)\int_0^\infty e^{-2t(1-\cos\theta)}|\Psi(t)|dtd\theta \\
& \leq C \|\Psi\|_{L^\infty(0,\infty )}\int_{-\pi}^\pi |\mathcal F_{\mathbb{Z}}(f)(\theta)|d\theta\leq C\|\Psi\|_{L^\infty (0,\infty )}\|\mathcal F_{\mathbb{Z}}(f)\|_{L^2(-\pi,\pi)}\\
&\leq C\|\Psi\|_{L^\infty(0,\infty )}\|f\|_{\ell^2(\mathbb Z)}.
\end{align*}
By using (\ref{eq5.1}) we get
\begin{align*}
\int_{-\pi}^\pi \mathfrak{m}(2(1-\cos \theta))\mathcal F_{\mathbb{Z}}(f)(\theta)e^{-in\theta}d\theta&=2\int_0^\infty\Psi(t)\int_{-\pi}^\pi \mathcal F_{\mathbb{Z}}(f)(\theta)(1-\cos \theta)e^{-2t(1-\cos \theta )}e^{-in\theta}d\theta dt\\
& = -\int_0^\infty\Psi(t)\partial_t\int_{-\pi}^\pi \mathcal F_{\mathbb{Z}}(f)(\theta)e^{-2t(1-\cos \theta)}e^{-in\theta}d\theta dt \\
& = -2\pi\int_0^\infty\Psi(t)\partial_t\mathcal F_{\mathbb{Z}}^{-1}[\mathcal F_{\mathbb{Z}}(f)\mathcal F_{\mathbb{Z}}(G(\cdot,t))](n) dt \\
& =-2\pi\int_0^\infty\Psi(t)\partial_t\sum_{m\in\mathbb Z}G(n-m,t)f(m)dt,\quad n\in\mathbb Z.
\end{align*}
Hence, if $f(m)=0$, for every $m\in \mathbb{Z}$ with $|m|>m_0$, for certain $m_0\in\mathbb N$, we get 
$$T_{\mathfrak{m}}(f)(n)=\sum_{\ell\in\mathbb Z}K_{\mathfrak m}(n-\ell)f(\ell),\;\;\;n\in\mathbb Z,$$
where $\displaystyle K_{\mathfrak m}(n)=-\int_0^\infty \Psi(t)\partial_tG(n,t)dt$, $n\in\mathbb Z$.

According to Lemma \ref{Lem2.3} we have that
$$
K_{\mathfrak{m}}(n)=-\frac{1}{\pi}\int_0^\infty \Psi (t)\partial _tH_0(n,t)dt,\quad n\in \mathbb{Z},
$$
where $H_0$ is the function given in Proposition \ref{derivadaHk} for $k=0$. Then, from \eqref{eq4.4} for $k=0$ we see that
\begin{equation}\label{eq5.2}
|K_{\mathfrak m}(n)|\leq\frac{C}{|n|},\quad n\in\mathbb Z,\;n\neq 0.
\end{equation}
Also we we can write
$$
K_{\mathfrak m}(n+\ell)-K_{\mathfrak m}(n)=\int_n^{n+\ell}\partial_zJ(z)dz,$$
where $\displaystyle J(z)=\frac{1}{\pi}\int_0^\infty \Psi(t)\partial_tH_1(z,t) dt$, $z>0$. Here $H_1$ is the function in Proposition \ref{derivadaHk} for $k=1$.
Again by \eqref{eq4.4} (for $k=1$) we deduce that
\begin{equation}\label{eq5.3}
|K_{\mathfrak m}(n+\ell)-K_{\mathfrak m}(n)|\leq C\left|\int_{n}^{n+\ell}\frac{dz}{z^2}\right|\leq C\frac{|\ell|}{n^2},\quad n,\ell\in\mathbb Z,\;n\cdot\ell> 0.
\end{equation}

We have seen that $T_{\mathfrak m}$ is a Calder\'on-Zygmund singular integral operator. Hence, $T_{\mathfrak m}$ can be extended from $\ell^2(\mathbb Z)\cap\ell^p(\mathbb Z,\omega)$ to $\ell^p(\mathbb Z,\omega)$ as a bounded operator from $\ell^p(\mathbb Z,\omega)$ into
\begin{enumerate}
\item[(a)] $\ell^p(\mathbb Z,\omega)$, for every $1<p<\infty$ and $\omega\in A_p(\mathbb Z)$,
\item[(b)] $\ell^{1,\infty}(\mathbb Z,\omega)$, when $p=1$ and $\omega\in A_1(\mathbb Z)$.
\end{enumerate}
Also, it can be shown that $T_{\mathfrak m}$ defines a bounded operator from $\mathcal H^p(\mathbb Z)$ to $\ell^p(\mathbb Z)$, for every $1/2<p\leq 1$, by using \eqref{eq5.3} and for $0<p\leq 1/2$ by means of \eqref{eq4.4} (see the argument in the proof of Theorem \ref{Th1.2} for $W_*$ and $g$).

Next we are going to see that $T_{\mathfrak m}$ defines a bounded operator from $\mathcal H^p(\mathbb Z)$ into itself, $0<p\leq 1$. In order to do this we will use molecules.

Discrete molecules were considered in \cite{KS} (see also \cite{Ko}). Let $0<p\leq 1<q\leq\infty$, and $\alpha >1/p-1/q$. A complex sequence $M$ is said to be a $(H,p,q,\alpha )$-molecule centered in $n_0\in\mathbb Z$ when the following conditions are satisfied.
\begin{enumerate}
\item[(i)] $N_{p,q,\alpha }(M)=\|M\|_q^{1-\theta}\||\cdot -n_0|^\alpha M\|_q^\theta <\infty,\;\;\;\mbox{where}\;\theta=(1/p-1/q)/\alpha ,$
\item[(ii)] $\displaystyle\sum_{n\in\mathbb Z}n^jM(n)=0$, $j\in\mathbb N$, $j\leq E[1/p]-1$.
\end{enumerate}
Note that the absolute convergence of the series in (ii) follows from the condition (i). In \cite[Theorem 2]{KS} Hardy spaces $\mathcal H^p(\mathbb Z)$ were characterized by using molecules.

Let $0<p\leq 1$ and $k=E[1/p]$. To establish that $T_{\mathfrak m}$ defines a bounded operator from $\mathcal H^p(\mathbb Z)$ into itself we see that there exists $C>0$ such that if $\mathfrak b$ is a $(H,p,2)$-atom associated to $n_0\in\mathbb Z$ and $r_0\geq 1$, then $T_{\mathfrak m}(\mathfrak b)$ is a $(H,p,2,k)$-molecule centered in $n_0$ and $N_{p,2,k}(T_{\mathfrak m}(\mathfrak b))\leq C$.

Assume that $\mathfrak b$ is a $(H,p,2)$-atom such that the support of $\mathfrak b$ is contained in $B_\mathbb{Z}(n_0,r_0)$ and $\|\mathfrak b\|_2\leq \mu(B_\mathbb{Z}(n_0,r_0))^{1/2-1/p}$, where $n_0\in\mathbb Z$ and $r_0\geq 1$. It is sufficient to consider $n_0=0$ because $T_{\mathfrak m}$ is a convolution operator.

Since $T_{\mathfrak m}$ is a bounded operator from $\ell^2(\mathbb Z)$ into itself we have that
$$
\|T_{\mathfrak m}(\mathfrak b)\|_2\leq C\|\mathfrak b\|_2\leq C\mu(B_\mathbb{Z}(0,r_0))^{1/2-1/p}.
$$

On the other hand we can write
$$
\||\cdot|^kT_{\mathfrak m}(\mathfrak b)\|_2^2 =\left(\sum_{n\in B_\mathbb{Z}(0,2r_0)}+\sum_{n\not\in B_\mathbb Z(0,2r_0)}\right)|n|^{2k}|T_{\mathfrak m}(\mathfrak b)(n)|^2=J_1+J_2.
$$

We have that
$$
J_1\leq Cr_0^{2k}\|\mathfrak b\|_2^2\leq Cr_0^{2k}\mu(B_\mathbb{Z}(0,r_0))^{1-2/p}\leq C\mu(B_\mathbb{Z} (0,r_0))^{2k+1-2/p}.
$$

To estimate $J_2$ we proceed as in the proof of Theorem \ref{Th1.2} for $W_*$ and $g$ by using \eqref{eq4.4}. We obtain 
\begin{align*}
\sum_{n\not \in B_\mathbb Z(0,2r_0)}|n|^{2k}|T_{\mathfrak m}(\mathfrak b)(n)|^2&\leq C\sum_{n\not \in B_\mathbb Z(0,2r_0)}|n|^{2k}\left(\sum_{m\in B_\mathbb{Z}(0,r_0)}|\mathfrak b(m)|\frac{|m|^k}{|n|^{k+1}}\right)^2 \\
& \leq C\sum_{n\not \in B_\mathbb Z(0,2r_0)}\frac{1}{|n|^2}\|\mathfrak b\|_2\sum_{m\in B_\mathbb{Z}(0,r_0)}|m|^{2k} \\
& \leq \frac{C}{r_0}\mu(B_\mathbb{Z}(0,r_0))^{1-2/p}r_0^{2k+1}\leq C\mu(B_\mathbb{Z}(0,r_0))^{2k+1-2/p}.
\end{align*}
By combining the above estimates we get
$$\||\cdot|^kT_{\mathfrak m}(\mathfrak b)\|_2\leq C\mu(B_\mathbb{Z}(0,r_0))^{k+1/2-1/p}.$$
We conclude that
$$
N_{p,2,k}(T_{\mathfrak m}(\mathfrak b))\leq C\mu(B_\mathbb{Z}(0,r_0))^{(1/2-1/p)(1-\theta )+(k+1/2-1/p)\theta}=C,
$$
where $\theta =(1/p-1/2)/k$.

Next we prove that
\begin{equation}\label{momentos}
\sum_{n\in \mathbb{Z}}n^jT_\mathfrak{m}\mathfrak{b}(n)=0,\quad j=0,...,k-1.
\end{equation}

According to \cite[p. 303]{KS}, $\sum_{n\in\mathbb Z}|n|^j|T_{\mathfrak m}(\mathfrak b)(n)|<\infty$, for every, $j\in\mathbb N$, $j\leq k-1$. Since $T_{\mathfrak m}(\mathfrak b)\in\ell^1(\mathbb Z)$, $\mathcal F_\mathbb{Z}(T_{\mathfrak m}(\mathfrak b))$ is a continuous function in $(-\pi,\pi)$. Also, $T_{\mathfrak m}(\mathfrak b)\in\ell^2(\mathbb Z)$ and  $\mathcal F_\mathbb{Z}(T_{\mathfrak m}(\mathfrak b))(\theta)=\mathfrak m(2(1-\cos \theta))\mathcal F_\mathbb{Z}(\mathfrak b)(\theta)$, a.e. $\theta\in(-\pi,\pi)$. Since $\mathfrak m(2(1-\cos \theta))\mathcal F_\mathbb{Z}(\mathfrak b)(\theta)$ is a continuous function in $(-\pi,\pi)\setminus\{0\}$, we have that
$$
\mathcal F_\mathbb{Z}(T_{\mathfrak m}(\mathfrak b))(\theta)=\mathfrak m(2(1-\cos \theta))\mathcal F_\mathbb{Z}(\mathfrak b)(\theta),\quad \theta\in(-\pi,\pi)\setminus\{0\}.
$$
Then, by taking into account that $\mathfrak m$ is a bounded function in $(0,\infty)$ and that $\mathcal F_\mathbb{Z}(\mathfrak b)(0)=\sum_{n\in\mathbb Z}\mathfrak b(n)=0$, we get
$$
\lim_{\theta\rightarrow 0}\mathcal F_\mathbb{Z}(T_{\mathfrak m}(\mathfrak b))(\theta)=0,
$$
and then we get \eqref{momentos} when $j=0$, because $\lim_{\theta\rightarrow 0}\mathcal F_\mathbb{Z}(T_{\mathfrak m}(\mathfrak b))(\theta)=\sum_{n\in \mathbb{Z}}T_\mathfrak{m}(\mathfrak{b})(n)$.

Let now $j\in\mathbb N$, $1\leq j\leq k-1$.

Since $\sum_{n\in\mathbb Z}|T_{\mathfrak m}(\mathfrak b)(n)||n|^j<\infty$, it follows that $\mathcal F_\mathbb{Z}(T_{\mathfrak m}(\mathfrak b))\in C^j(-\pi,\pi)$ with
$$
\frac{d^j}{d\theta^j}\mathcal F_\mathbb{Z}(T_{\mathfrak m}(\mathfrak b))(\theta)=i^j\sum_{n\in \mathbb Z}n^j T_{\mathfrak m}(\mathfrak b)(n)e^{in\theta}, \quad \theta\in(-\pi,\pi),
$$
and 
$$
\sum_{n\in \mathbb{Z}}n^jT_\mathfrak{m}(\mathfrak{b})(n)=(-i)^j\lim_{\theta \rightarrow 0}\frac{d^j}{d\theta^j}\mathcal F_\mathbb{Z}(T_{\mathfrak m}(\mathfrak b))(\theta).
$$
On the other hand, we have that
$$
\frac{d^j}{d\theta^j}\mathcal F_\mathbb{Z}(T_{\mathfrak m}(\mathfrak b))(\theta)=\sum_{\ell=0}^j{{j}\choose{\ell}}\frac{d^{j-\ell}}{d\theta^{j-\ell}}\mathfrak m(2(1-\cos \theta))\frac{d^\ell}{d\theta^\ell}\mathcal F_\mathbb{Z}(\mathfrak b)(\theta),\quad \theta\in(-\pi,\pi)\setminus\{0\}.
$$
It is not hard to see that, for every $r\in\mathbb N$, there exists $C>0$ such that
$$
\left|\frac{d^r}{d\theta^r}\mathfrak m(2(1-\cos \theta))\right|\leq \frac{C}{|\theta|^r},\;\;\;\theta\in(-\pi,\pi)\setminus\{0\}.
$$
By using Taylor theorem, since $\frac{d^\ell}{d\theta^\ell}\mathcal F_\mathbb{Z}(\mathfrak b)(0)=0$, $\ell\in\mathbb N$, $0\leq\ell\leq k-1$, we get, for a certain $C>0$,
$$\left|\frac{d^\ell}{d\theta^\ell}\mathcal F_\mathbb{Z}(\mathfrak b)(\theta)\right|\leq C|\theta|^{k-\ell},\;\;\;\theta\in(-\pi,\pi)\setminus\{0\}\;\mbox{and}\;\ell =0,1,...,k-1.$$

We conclude that $\lim_{\theta\rightarrow 0}\frac{d^j}{d\theta^j}\mathcal F_\mathbb{Z}(T_{\mathfrak m}(\mathfrak b))(\theta)=0$ and hence, $\sum_{n\in\mathbb Z}n^j T_{\mathfrak m}(\mathfrak b)(n)=0.$
Thus, we have proved that the operator $T_{\mathfrak m}$ is bounded from $\mathcal H^p(\mathbb Z)$ into itself for every $p\in(0,1]$.

\def\cprime{$'$} \def\ocirc#1{\ifmmode\setbox0=\hbox{$#1$}\dimen0=\ht0
  \advance\dimen0 by1pt\rlap{\hbox to\wd0{\hss\raise\dimen0
  \hbox{\hskip.2em$\scriptscriptstyle\circ$}\hss}}#1\else {\accent"17 #1}\fi}

%\bibliographystyle{siam}
%\bibliography{references}

\begin{thebibliography}{10}

\bibitem{Bo}
{\sc S.~Boza}, {\em Factorization of sequences in discrete {H}ardy spaces},
  Studia Math., 209 (2012), pp.~53--69.

\bibitem{BC1}
{\sc S.~Boza and M.~a.~J. Carro}, {\em Discrete {H}ardy spaces}, Studia Math.,
  129 (1998), pp.~31--50.

\bibitem{BC3}
\leavevmode\vrule height 2pt depth -1.6pt width 23pt, {\em Convolution
  operators on discrete {H}ardy spaces}, Math. Nachr., 226 (2001), pp.~17--33.

\bibitem{BC2}
\leavevmode\vrule height 2pt depth -1.6pt width 23pt, {\em Hardy spaces on
  {$\Bbb Z^N$}}, Proc. Roy. Soc. Edinburgh Sect. A, 132 (2002), pp.~25--43.

\bibitem{CF}
{\sc G.~Chen and G.~Fang}, {\em Discrete characterization of the
  {P}aley-{W}iener space with several variables}, Acta Math. Appl. Sinica
  (English Ser.), 16 (2000), pp.~396--404.

\bibitem{CGRTV}
{\sc O.~Ciaurri, T.~A. Gillespie, L.~Roncal, J.~L. Torrea, and J.~L. Varona},
  {\em Harmonic analysis associated with a discrete {L}aplacian}, J. Anal.
  Math., 132 (2017), pp.~109--131.

\bibitem{CRW}
{\sc R.~R. Coifman, R.~Rochberg, and G.~Weiss}, {\em Factorization theorems for
  {H}ardy spaces in several variables}, Ann. of Math. (2), 103 (1976),
  pp.~611--635.

\bibitem{CW}
{\sc R.~R. Coifman and G.~Weiss}, {\em Extensions of {H}ardy spaces and their
  use in analysis}, Bull. Amer. Math. Soc., 83 (1977), pp.~569--645.

\bibitem{Eo}
{\sc C.~Eoff}, {\em The discrete nature of the {P}aley-{W}iener spaces}, Proc.
  Amer. Math. Soc., 123 (1995), pp.~505--512.

\bibitem{FaBruno}
{\sc F.~Fa\`a~di Bruno}, {\em Sullo sviluppo delle funzioni}, Ann. Sci. Mat.
  Fis., Roma, 6 (1885), pp.~479--480.

\bibitem{GCRF}
{\sc J.~Garc\'{i}a-Cuerva and J.~L. Rubio~de Francia}, {\em Weighted norm
  inequalities and related topics}, vol.~116 of North-Holland Mathematics
  Studies, North-Holland Publishing Co., Amsterdam, 1985.
\newblock Notas de Matem\'{a}tica [Mathematical Notes], 104.

\bibitem{HMW}
{\sc R.~Hunt, B.~Muckenhoupt, and R.~Wheeden}, {\em Weighted norm inequalities
  for the conjugate function and {H}ilbert transform}, Trans. Amer. Math. Soc.,
  176 (1973), pp.~227--251.

\bibitem{KS}
{\sc Y.~Kanjin and M.~Satake}, {\em Inequalities for discrete {H}ardy spaces},
  Acta Math. Hungar., 89 (2000), pp.~301--313.

\bibitem{Ko}
{\sc Y.~Komori}, {\em The atomic decomposition of molecule on discrete {H}ardy
  spaces}, Acta Math. Hungar., 95 (2002), pp.~21--27.

\bibitem{Le}
{\sc N.~N. Lebedev}, {\em Special functions and their applications}, Dover
  Publications, Inc., New York, 1972.
\newblock Revised edition, translated from the Russian and edited by Richard A.
  Silverman, Unabridged and corrected republication.

\bibitem{LL}
{\sc M.-Y. Lee and C.-C. Lin}, {\em The molecular characterization of weighted
  {H}ardy spaces}, J. Funct. Anal., 188 (2002), pp.~442--460.

\bibitem{Mi}
{\sc A.~Miyachi}, {\em Weak factorization of distributions in {$H^{p}$}
  spaces}, Pacific J. Math., 115 (1984), pp.~165--175.

\bibitem{PBM}
{\sc A.~P. Prudnikov, Y.~A. Brychkov, and O.~I. Marichev}, {\em Integrals and
  series. {V}ol. 1}, Gordon \& Breach Science Publishers, New York, 1986.
\newblock Elementary functions, Translated from the Russian and with a preface
  by N. M. Queen.

\bibitem{RRT}
{\sc J.~L. Rubio~de Francia, F.~J. Ruiz, and J.~L. Torrea}, {\em
  Calder\'{o}n-{Z}ygmund theory for operator-valued kernels}, Adv. in Math., 62
  (1986), pp.~7--48.

\bibitem{StLP}
{\sc E.~M. Stein}, {\em Topics in harmonic analysis related to the
  {L}ittlewood-{P}aley theory}, Annals of Mathematics Studies, No. 63,
  Princeton University Press, Princeton, N.J.; University of Tokyo Press,
  Tokyo, 1970.

\end{thebibliography}

\end{document}